\documentclass[12pt, leqno, british]{amsart}
\usepackage[
	style=alphabetic,
	citestyle=alphabetic,
	backend=biber
]{biblatex}

\usepackage{harmony}
\usepackage{a4, amsmath}
\usepackage{mathtools}
\usepackage{amssymb}
\usepackage{epsfig}
\usepackage{graphics}
\usepackage{amsthm, amscd, mathdots}
\usepackage{enumerate}
\usepackage{url}
\usepackage{cancel}
\usepackage{hyperref} 
\usepackage{csquotes}
\usepackage{color}
\usepackage{datetime}
\usepackage{stmaryrd}
\usepackage{ifthen}

\usepackage[pass]{geometry}

\usepackage[all]{xy}
\usepackage{cleveref}

\theoremstyle{definition}
\newtheorem{defi}{Definition}[section]

\theoremstyle{plain}

\newtheorem{lem}[defi]{Lemma}
\newtheorem{thm}[defi]{Theorem}
\newtheorem{cor}[defi]{Corollary}

\newtheorem{ques}[defi]{Question}
\newtheorem*{stel*}{Theorem}
\theoremstyle{remark}
\newtheorem{opm}[defi]{Remark}
\newtheorem{vb}[defi]{Example}

\newcommand{\La}{\mathcal{L}}
\newcommand{\nat}{\mathbb{N}}
\newcommand{\zz}{\mathbb{Z}}
\newcommand{\rr}{\mathbb{R}}
\newcommand{\ff}{\mathbb{F}}

\newcommand{\qq}{\mathbb{Q}}

\newcommand{\Lar}{{\La_{\rm ring}}}
\newcommand{\LaP}{\La_{\rm {Ph}}}

\newtheorem*{prop*}{Proposition}

\newcommand{\mc}{\mathcal}

\newcommand{\mf}{\mathfrak}
\newcommand{\mbb}{\mathbb}
\renewcommand{\div}{\operatorname{div}}

\newcommand{\sep}[3][]{
\ifx &#1&
{{#2}{(\SePa#3)}}
\else
{{#2}_{(#3, #1)}}
\fi
}

\DeclareMathOperator{\charac}{char}

\DeclareMathOperator{\Th}{Th}

\DeclareMathOperator{\Trd}{Trd}

\DeclareMathOperator{\Nrd}{Nrd}

\DeclareMathOperator{\red}{\mathsf{red}}

\DeclareMathOperator{\Aut}{Aut}

\newcommand{\qraq}{\quad\rightarrow\quad}

\title[Decidability of polynomial equations over function fields]{Decidability of polynomial equations over function fields in positive characteristic}
\author{Nicolas Daans}
\date{\today}
\address{KU Leuven, Department of Mathematics, Celestijnenlaan 200B, 3001 Heverlee, Belgium}
\email{nicolas.daans@kuleuven.be}
\address{Université de Mons, Département de Mathematique, Place du Parc 20, 7000 Mons, Belgium}
\email{nicolas.daans@umons.ac.be}
\keywords{Hilbert's 10th Problem, algebraic function field, Frobenius automorphism, definable valuation}
\subjclass{Primary: 12L05. Secondary: 03B25, 11R58, 11U05, 12F20.}

\begin{document}
\begin{abstract}
Let $K$ be a field of positive characteristic with no algebraically closed subfield.
Let $F$ be a function field over $K$ and $t \in F$ transcendental over $K$.
Refining a result of Eisentr{\"a}ger and Shlapentokh \cite{EisShlap17}, we show that there is no algorithm which, on input a polynomial $f \in \zz[t][X_1, \ldots, X_n]$, determines whether $f$ has a zero in $F^n$.
To this end, we revisit and partially extend several recent results from the literature on existential definability in function fields.
\end{abstract}
\maketitle

\section{Introduction}
Give a computer any multivariate polynomial with integer coefficients $f \in \zz[X_1, \ldots, X_n]$, and it can tell you whether or not it has a zero in $\rr^n$ - at least, given enough time and computational resources.
This follows from Tarski's foundational work around 1950, in which he developed a decision procedure for questions on the arithmetic of the field of real numbers $\rr$ \cite{Tar51}.
Later work by Ax and Kochen yielded a similar decision procedure for the arithmetic of fields of $p$-adic numbers; in particular, for any prime number $p$, one can decide algorithmetically whether polynomials with integer coefficients have a zero over $\qq_p$ \cite{Ax-Kochen-DiophantineLocalIII} (see also the earlier work of Nerode on $\zz_p$ \cite{Nerode}).

It is believed that the situation may look very differently when replacing the field of real numbers $\rr$ with a number field, like the field of rational numbers $\qq$.
But while it has been known since the work of Robinson (building on G{\"o}del's Incompleteness Theorems) that the full first-order of the field $\qq$ (or any number field) is undecidable \cite{Rob49,Rob65}, we still do not know whether the same holds for its positive-existential theory, i.e. whether there can be an algorithm which can determine the solvability of polynomial equations over $\qq$.

Unable for now to answer this question for number fields, we might turn our attention to their positive-characteristic counterparts: global function fields.
For example, denoting for a field $K$ by $K(t)$ the rational function field over $K$, we may ask:
\begin{ques}\label{Q:H10-ff_p(T)}
Let $K$ be a finite field.
Does there exist an algorithm which, upon being given as input a natural number $n \in \nat$ and $f \in \zz[X_1, \ldots, X_n]$, decides whether there exists $x \in K(t)^n$ such that $f(x) = 0$?
\end{ques}
This question, too, remains unresolved to this day. Anscombe and Fehm recently established a positive answer to a variant of \Cref{Q:H10-ff_p(T)}, with $K(t)$ replaced by the field of formal Laurent series $K(\!(t)\!)$ \cite{AF16}, whereas Tyrrell-Nic Dhonncha proved undecidability of a slightly more general decision problem in \cite{H10Tyrrell-NicDohnncha}.

So far, we have only discussed solvability of polynomial equations with integer coefficients.
We can replace the ring of integers $\zz$ by a more general commutative ring $R$, but before we can sensibly phrase questions about algorithms taking elements of $R$ as input, we need to fix a recursive presentation of $R$.
In this paper, we will restrict our attention to the situation where $R$ is finitely generated.
In this case, $R$ naturally has a recursive presentation and any two recursive presentations of $R$ are equivalent (see for example \cite[Section 2]{Handbook-CompRingsFields}, where this property is called computable stability), whence one can unambiguously speak about algorithms which take a polynomial $f \in R[X_1, \ldots, X_n]$ as input.
Given a commutative ring $K$ endowed with a morphism $R \to K$ (often implicit when clear from the context), let us say that \emph{$\Th_{\exists}(K, R)$ is decidable} if there exists an algorithm which, upon receiving as input $n \in \nat$ and $f \in R[X_1, \ldots, X_n]$, decides whether there exists $x \in K^n$ with $f(x) = 0$.
Otherwise, we say that \emph{$\Th_{\exists}(K, R)$ is undecidable}.
The results we discussed so far can be summarised by saying that $\Th_{\exists}(\rr, \zz)$, $\Th_{\exists}(\qq_p, \zz)$ and $\Th_{\exists}(\ff_p(\!(t)\!), \zz)$ are decidable, and \Cref{Q:H10-ff_p(T)} asks whether the same holds for $\Th_{\exists}(\ff_p(t), \zz)$.
In the literature, the question of decidability of $\Th_\exists(K, R)$ is often referred to as \emph{Hilbert's Tenth Problem over $K$ (with coefficients in $R$)}: indeed Hilbert had asked for a decision algorithm for $\Th_\exists(\zz, \zz)$, but the work of Davis, Putnam, Robinson and Matiyasevich showed that $\Th_\exists(\zz, \zz)$ is undecidable \cite{Mat70}.

In contrast to the open \Cref{Q:H10-ff_p(T)}, there is the following result, established first by Pheidas in odd characteristic and later by Videla in characteristic $2$:
\begin{thm}[Pheidas \cite{PheidasHilbert10}, Videla \cite{Videla}]\label{TI:Pheidas}
Let $K$ be a finite field. Then $\Th_\exists(K(t), \zz[t])$ is undecidable.
\end{thm}
The methods of Pheidas and Videla were scrutinized and generalized in subsequent works, first by Kim and Roush to rational function fields over certain infinite base fields \cite{KimRoush-charp}, and then by the works of Shlapentokh and Eisentr{\"a}ger, first to arbitrary global function fields \cite{Shl96,Eis03}, then to more general function fields of positive characteristic, eventually leading to the following conclusion.
A finitely generated transcendental field extension $F/K$ is called a \emph{function field}.
\begin{thm}[Eisentr{\"a}ger-Shlapentokh, \cite{EisShlap17}]\label{TI:ES}
Let $K$ be a field of characteristic $p > 0$ not containing an algebraically closed subfield.
Let $F/K$ be a function field.
There exists a finitely generated subring $R$ of $F$ such that $\Th_\exists(F, R)$ is undecidable.
\end{thm}
The proof of \Cref{TI:ES} implicitly constructs the finitely generated ring $R$, but leaves open to what degree the result depends on this specific choice of this ring of coefficients.
In this paper, we revisit this result, and show that one can take for $R$ any subring of $F$ containing a transcendental element:
\begin{thm}[see \Cref{T:ES-constants}]\label{TI:ES-constants}
Let $K$ be a field of characteristic $p > 0$ not containing an algebraically closed subfield.
Let $F/K$ be a function field and let $t \in F$ be transcendental over $K$.
Then $\Th_\exists(F, \zz[t])$ is undecidable.
\end{thm}
The motivation for this is two-fold.
Firstly, we must remark that in general, the decidability of $\Th_\exists(K, R)$ really may change when the ring $R$ is replaced by another finitely generated ring.
For example, $\Th_\exists(\rr(t), \zz)$ is decidable (one may take exactly the same decision algorithm as for $\Th_\exists(\rr, \zz)$), but $\Th_\exists(\rr(t), \zz[t])$ is undecidable \cite{DenefDiophantine}; this gives an indication that \Cref{TI:Pheidas} may not tell us that much about what the answer to \Cref{Q:H10-ff_p(T)} should be.
Perhaps even more surprisingly, while $\Th_\exists(\rr, \zz)$ is decidable, this is no longer the case when a general real number is added to the ring of coefficients, as the following example (pointed out by Philip Dittmann) illustrates.
\begin{vb}
Let $r$ be an uncomputable real number.
Then $\Th_{\exists}(\rr, \zz[r])$ is undecidable.
Indeed, by definition of an uncomputable real number, there can be no algorithm which decides whether for given non-zero $a, b \in \zz$ one has $\frac{a}{b} \leq r$, but the relation $\frac{a}{b} \leq r$ is equivalent to the polynomial $X^2 - br + a \in \zz[r][X]$ having a root in $\rr$.
\end{vb}
This shows that \Cref{TI:ES-constants} strengthens \Cref{TI:ES} in a non-trivial way.

Our second motivation is methodological.
While at a high level our proof technique follows that of \cite{PheidasHilbert10,Videla,Shl96,EisShlap17},
we deviate from their arguments in some key ways.
Furthermore, as will be clear from the outline below, several intermediate definability results can be proven for function fields over arbitrary base fields, and as such may in the future help to establish new (un)decidability results for function fields as well.
We mention the following open question in this context, complementary to \Cref{Q:H10-ff_p(T)}.
Here, $\overline{\ff_p}$ denotes the algebraic closure of $\ff_p$.
\begin{ques}
Is $\Th_\exists(\overline{\ff_p}(t), \zz[t])$ decidable?
\end{ques}
Here, one may remark that undecidability of $\Th_\exists(\overline{\ff_p}(t_1, t_2), \zz[t_1, t_2])$ was shown by Kim and Roush for $p > 2$ \cite{KimRoushCt1t2}, later generalized to include finite extensions of $\overline{\ff_p}(t_1, t_2)$ as well by Eisentr{\"a}ger \cite{EisFunctionFieldsPositiveCharacteristic}.
Recently Becher, Dittmann and the author obtained more examples of function fields $F/K$ where $K$ contains $\overline{\ff_p}$ and $\Th_\exists(F, R)$ is shown to be undecidable for some well-chosen ring of coefficients $R$ \cite{BDD}, relying partially on the type of reduction techniques discussed in the present paper.

We now give an outline of the structure and some of the main ideas of this paper.
We shall find it convenient to use the framework of first-order logic.
Thus, in \Cref{S:definability}, the main tools we will need are briefly recalled.

We will then give a proof of the following existential definability result, which sharpens the main result of \cite[Section 5]{EisShlap17}.
Here, $\Lar$ denotes the first-order language of rings (see \Cref{S:definability} for a definition).

\begin{thm}[see \Cref{C:define-Pn}]\label{TI:Frobenius}
Let $K$ be a field of characteristic $p > 0$, $F/K$ a function field, $t \in F$ any element transcendental over $K$.
There exists a positive-existential $\Lar$-formula $\pi(x, y, z)$ such that for all $f, g \in F$ one has
$$ F \models \pi(f, g, t) \quad\Leftrightarrow\quad \exists s \in \nat : f = g^{p^s}.$$
\end{thm}
The formula $\pi$ in the above Theorem depends on the specific field $F$ only via an upper bound on the genus and some $n \in \nat$ for which $t^{1/p^n} \not\in F$, see \Cref{T:define-Pn} for a precise statement.

While our proof of \Cref{TI:Frobenius} shares ideas with the one given in \cite{EisShlap17}, it actually follows closely that of Pasten's \cite[Theorem 1.5]{Pasten_FrobeniusOrbits}, where \Cref{TI:Frobenius} is shown in the case $p > 2$ under a mildly stronger hypothesis on $t$.
In \Cref{S:p=2} we thus develop the algebraic background specifically for the case $p=2$, which caused issues in the proof of \cite[Theorem 1.5]{Pasten_FrobeniusOrbits}.
In \Cref{S:def-Frobenius} we prove \Cref{C:define-Pn} for arbitrary $p>0$, essentially redoing the proof of \cite{Pasten_FrobeniusOrbits}.

In \Cref{S:undecidability}, we explain how the aforementioned result reduces the problem of establishing undecidability of $\Th_\exists(F, R)$ for a function field $F$ to showing the positive-existential definability of a valuation-like predicate in the language of rings.
This idea goes back to \cite{PheidasHilbert10} and was even already analyzed more abstractly in \cite[Lemma 1.5]{Shl96} and \cite[Section 2]{PheidasZahidi_survey}, see also \cite[Section 2]{EisShlap17}.
We include it here in our set-up for the reader's convenience, and to emphasize the (in)dependence of specific parameters.

Finally, in \Cref{S:final} the proof of \Cref{TI:ES-constants} is completed by proving the positive-existential definability of the valuation-like predicate for function fields over fields of characteristic $p>0$ not containing an algebraically closed base field (\Cref{T:defining-valuation}).
Here we deviate substantially from the presentation of \cite{EisShlap17} and instead take inspiration from \cite{Dit17}, although several additional delicate reductions are needed to prevent the introduction of new parameters.

\subsection*{Acknowledgements and funding}
The author would like to thank Philip Dittmann for numerous helpful discussions over the years on topics related to the research presented in this paper, in particular regarding the proof of \Cref{T:define-Pn} and some of the methodological ideas used in \Cref{S:final}, for providing several helpful pointers to the literature, and for his instructive and elaborate feedback on an earlier draft of this paper.

The author gratefully acknowledges financial support from \emph{Research Foundation - Flanders (FWO)}, fellowship 1208225N.

\section{Preliminaries on definability}\label{S:definability}
We will phrase some of our results using the language of first-order logic.
We refer to \cite[Sections II and III]{Ebb94} for an introduction to first-order logic, and will only establish some notation and conventions here.

Consider a first-order language $\La$, consisting of a set of \emph{constant symbols}, and for each natural number $n \geq 1$, a set of \emph{$n$-ary function symbols} and a set of \emph{$n$-ary relation symbols} (all these sets together form what is called the \emph{symbol set} of the language in \cite[14]{Ebb94}).
An \emph{$\La$-structure} $\mc A$ consists of a set $A$ (the \emph{domain} of $\mc A$), for each constant symbol $c$ of $\La$ an element $c^{\mc A}$ of $A$, for each $n$-ary relation symbol $R$ of $\La$ an $n$-ary relation $R^{\mc A} \subseteq A^n$, and for each $n$-ary function symbol $F$ of $\La$ an $n$-ary function $F^{\mc A} : A^n \to A$.
In practice, we shall often abuse notation and use the same letter for a structure and its domain when there is little risk of confusion.

When introducing an $\La$-formula $\varphi$, we may write $\varphi(x_1, \ldots, x_n)$ for distinct variable symbols $x_1, \ldots, x_n$ to indicate that no variable other than $x_1, \ldots, x_n$ occurs freely in $\varphi$, and to fix an order of these variables. When $t_1, \ldots, t_n$ are $\La$-terms (like variable symbols or constant symbols), we may then later write $\varphi(t_1, \ldots, t_n)$ to mean the $\La$-formula obtained by substituting each free occurrence of $x_i$ by $t_i$ (as described in \cite[Section III.8]{Ebb94}).

An $\La$-formula without free variables is called an \emph{$\La$-sentence}.
Given an $\La$-structure $A$ and an $\La$-sentence $\varphi$, we write $A \models \varphi$ to mean that \emph{$\varphi$ is satisfied in $A$} (as defined in \cite[Section III.3]{Ebb94}).

A \emph{positive-existential $\La$-formula} is a formula which is logically equivalent to a formula of the form $\exists x_1 \ldots \exists x_n \psi$ where $\psi$ is a positive quantifier-free $\La$-formula, i.e.~a formula built up from atomic formulas using disjunction ($\vee$) and conjunction ($\wedge$), but no negation.
One sees that a finite conjunction or disjunction of positive-existential $\La$-formulas is again a positive-existential $\La$-formula.

Given an $\La$-structure $R$, some $n \geq 1$ and a set $S \subseteq R^n$, we say that an $\La$-formula $\varphi(x_1, \ldots, x_n)$ \emph{defines $S$} if
$$ S = \lbrace (a_1, \ldots, a_n) \in R^n \mid R \models \varphi(a_1, \ldots, a_n) \rbrace.$$
We will say that $S$ is \emph{$\La$-definable} if it is definable by an $\La$-formula.
If this formula can be chosen to be positive-existential, we call $S$ \emph{positive-existentially $\La$-definable.}

When $R$ is an $\La$-structure and $S \subseteq R$ a subset, we denote by $\La(S)$ the first-order language which expands $\La$ by a new constant symbol $c_r$ for each $r \in S$, and we interpret $c_r^R = r$.
We shall again often abuse notation and just write $r$ instead of $c_r$.

We denote by $\Lar$ the first-order language of rings, consisting of binary operation symbols $+, -, \cdot$ and constant symbols $0, 1$, and we will always interpret fields as $\Lar$-structures in the obvious way.
In fact, all $\Lar$-structures which we will consider will be fields.
Given polynomials $F,G \in \zz[X_1, \ldots, X_n]$, we may find an atomic $\Lar$-formula $\phi_{F,G}(x_1, \ldots, x_n)$ such that for any field $K$ and $a_1, \ldots, a_n \in K$ we have $F(a_1, \ldots, a_n) = G(a_1, \ldots, a_n)$ if and only if $K \models \phi_{F, G}(a_1, \ldots, a_n)$.
As the precise choice of this $\phi_{F, G}$ does not matter as long as we only ever evaluate the formula in fields, we will simply write $F = G$ for any such choice of formula $\phi_{F, G}$.

We record for later use the following lemma about deducing positive-existential definability of subsets of fields from positive-existential definability of subsets of finite field extensions.
It is a variant of a well-known reduction trick; we spell out a specific version of it given our heightened focus on keeping track of used constant symbols.

\begin{lem}\label{L:interpretation-argument}
Consider an arbitrary field extension $L/K$ and a finite normal field extension $N/K$.
Let $M$ be a common overfield of $L$ and $N$ over $K$ which is generated by $L \cup N$.
If $S \subseteq M^n$ is positive-existentially $\Lar(K)$-definable over $M$, then $S \cap L^n$ is positive-existentially $\Lar(K)$-definable over $L$.
\end{lem}
\begin{proof}
Consider $\mc D = L \otimes_K N$, and denote by $\mc D^{\red}$ its reduction (i.e.~quotient modulo the nilradical). We observe the following properties:
\begin{enumerate}
\item Under the inclusion $N \to \mc D : x \mapsto 1 \otimes x$, a $K$-basis of $N$ gets mapped to an $L$-basis of $\mc D$.
In particular, $\mc D$ is a finite-dimensional $L$-algebra, and it has a basis whose structure constants all lie in $K$.
\item There exists $k \in \nat$ such that every nilpotent element $a$ of $\mc D$ satisfies $a^k = 0$.
\item There is a subset $G \subseteq \Aut(N/K)$ and a well-defined $L$-isomorphism
$ \Phi : \mc D^{\red} \to M^G : $
such that $\Phi(a \otimes x) = (a \sigma(x))_{\sigma \in G}$ for all $a \in L$, $x \in N$.
\end{enumerate}
The first property is immediate from standard properties of the tensor product. The second property follows because $\mc D$ is artinian and hence the nilradical is nilpotent (see e.g.~\cite[Proposition 8.4]{Ati69}).

For the third property: as $\mc{D}^{\red}$ is a reduced artinian ring, it is isomorphic as an $L$-algebra to a finite product of field extensions of $L$, obtained by considering all quotients of $\mc{D}^{\red}$ modulo maximal ideals.
Clearly at least one of these quotients is equal to $M$, but since $N/K$ is normal, $\Aut(N/K)$ acts transitively on the spectrum of $\mc D^{\red}$, from which the desired statement follows.

Now consider a positive-existential $\Lar(K)$-formula $\varphi(x_1, \ldots, x_n)$ such that $S = \lbrace (b_1, \ldots, b_n) \in M^n \mid M \models \varphi(b_1, \ldots, b_n) \rbrace$.
We observe that
\begin{displaymath}
S \cap L^n = \{ (b_1, \ldots, b_n) \in L^n \mid \mc D^{\red} \models \varphi(b_1 \otimes 1, \ldots, b_n \otimes 1) \}.
\end{displaymath}
Indeed, in view of the $L$-isomorphism $\mc{D}^{\red} \cong M^G$ from (3), there clearly exist $L$-homomorphisms $\mc{D}^{\red} \to M$ and $M \to \mc{D}^{\red}$, and validity of positive-existential formulas is preserved under homomorphisms.

The positive-existential $\Lar(K)$-definability of $S \cap L^n$ in $L$ then follows via the usual quantifier-free interpretation of the finite-dimensional algebra $\mc D$ (as $\Lar(K)$-structure) in $L$ (as $\Lar(K)$-structure), and the fact that there exists a natural number $k$ such that every nilpotent element $a \in \mc{D}$ satisfies $a^k = 0$ (so for $a, b \in \mc{D}$ we have $\overline{a} = \overline{b}$ in $\mc{D}^{\red}$ if and only if $(a-b)^k = 0$ in $\mc{D}$).
\end{proof}
Finally, consider a finite language $\La$ (i.e.~with a finite symbol set).
We can naturally enumerate the set of positive-existential $\La$-sentences.
Hence, given an $\La$-structure $R$, we can ask whether the \emph{positive-existential theory of $R$ is decidable}, i.e.~whether there can be an algorithm which takes as input a positive-existential $\La$-sentence $\varphi$ and outputs YES if $R \models \varphi$ and NO otherwise.
See \cite{HMLElementsOfRecursion,HMLDecidableTheories} for background on (un)decidable theories.

When $K$ is a field and $R$ either a finitely generated subring or a finitely generated subfield of $K$, say with generators $r_1, \ldots, r_n$, then every positive-existential $\Lar(R)$-sentence can be effectively transformed into an equivalent (with respect to the validity in $K$) positive-existential $\Lar(\{ r_1, \ldots, r_n \})$-sentence, by replacing constant symbols from $R$ by polynomial (respectively rational) expressions in $r_1, \ldots, r_n$, and possibly clearing denominators afterwards.
We say that the \emph{positive-existential $\Lar(R)$-theory of $K$ is decidable} if the positive-existential $\Lar(\{ r_1, \ldots, r_n \})$-theory of $K$ is decidable, in the sense explained in the previous paragraph. As alluded to in the Introduction, this does in fact not depend on the choice of generators for $R$ (see \cite[Section 2]{Handbook-CompRingsFields}).
Furthermore, it corresponds precisely to what was called decidability of $\Th_{\exists}(K, R)$ in the Introduction: indeed, given a positive-existential $\Lar(R)$-sentence $\varphi$, one can effectively find a polynomial defined over $R$ with the property that $K \models \varphi$ if and only if the polynomial has a zero over $K$ (see for example \cite[Remark 3.4]{DDF}).

\section{Function fields in characteristic $2$}\label{S:p=2}
In this section, we will consider \emph{function fields in one variable} $F/K$, i.e.~finitely generated field extensions of transcendence degree $1$.
We will consider the \emph{genus} of such function fields, as defined in \cite[Chapter II]{Chevalley-Functions}.

Function fields in one variable over $K$ arise as the function fields of integral algebraic curves over $K$; we write $K(C)$ for the function field of an integral curve $C$ defined over $K$.
If $K$ is perfect, then every function field in one variable over $K$ is the function field of a smooth projective curve over $K$ \cite[Proposition 7.3.13, in view of Corollary 4.3.33]{Liu}.

Our general reference in this section for background on smooth curves over perfect fields will be \cite{Silverman_EllCurves}.
Given such a smooth projective curve $C$ over a field $K$, we will consider elements of $K(C)$ as rational maps $C \to \mbb{P}_K^1$ where $\mbb{P}_K^1$ is the projective line over $K$.
Elements of $K(C)$ which induce constant maps correspond to elements which are algebraic over $K$.
We will also consider the \emph{genus} of such curves as described in \cite[Section II.5]{Silverman_EllCurves}.
The genus of a function field in one variable over a perfect base field is equal to the genus of any smooth projective curve of which it is the function field (see \cite[Appendix B.10]{Sti09}).

For integers $\mf g$, $d$ and $p$, define
$$ M(\mf g, d, p) = \left\lceil \frac{1}{d} \left( 4 \mf g + 12 + 8 \sum_{j=1}^{\lceil (d-1)/2 \rceil} p^j \right) \right\rceil.$$
The following statement was shown in \cite[Theorem 1.6]{Pasten_FrobeniusOrbits}.

\begin{thm}\label{thm:Frob-orbit-p}
Let $\mf g \geq 0$ and $d \geq 1$ be integers and let $p > 2$ be prime.
Set $M = M(\mf g, d, p)$.

Let $K$ be a field of characteristic $p$ and let $F$ be a function field in one variable of genus at most $\mf g$ defined over $K$.
Let $F_1, \ldots, F_M \in \ff_p[X]$ be pairwise coprime irreducible polynomials of degree $d$.
Let $f, g \in F$ not both constant.
The following are equivalent:
\begin{itemize}
\item There is $s \geq 0$ such that $f = g^{p^s}$ or $g = f^{p^s}$.
\item For all $j = 1, \ldots, M$ there exists $h_j \in F$ with $F_j(f)F_j(g) = h_j^2$.
\end{itemize}
\end{thm}
\noindent
The goal of this section is to prove the following variation for $p=2$.
\begin{thm}\label{thm:Frob-orbit-2}
Let $\mf g \geq 0$ and $d \geq 1$ be integers.
Set $M = M(\mf g, d, 2)$.

Let $K$ be a field of characteristic $2$ and let $F$ be a function field in one variable of genus at most $\mf g$ defined over $K$.
Let $F_1, \ldots, F_M \in \ff_2[X]$ be pairwise coprime irreducible polynomials of degree $d$.
Let $f, g \in F$ not both constant.
The following are equivalent:
\begin{itemize}
\item There is $s \geq 0$ such that $f = g^{2^s}$ or $g = f^{2^s}$.
\item For all $j = 1, \ldots, M$ there exists $h_j \in F$ with $F_j(f) + F_j(g) = F_j(f)F_j(g)(h_j^2 + h_j)$.
\end{itemize}
\end{thm}
For a field $K$ of characteristic $2$, let us denote the set of Artin-Schreier elements
$$ \mf a (K) = \lbrace h^2 + h \mid h \in K \rbrace.$$
We gather some basic observations in the following lemma, each of which is verified easily.
\begin{lem}\label{L:Frobenius-trick}
Let $K$ be a field of characteristic $2$. We have that
\begin{itemize}
\item $0 \in \mf a(K)$, and if $x, y \in \mf a(K)$, then also $x + y \in a(K)$,
\item for $x \in K$, we have $x \in \mf a (K)$ if and only if $x^2 \in \mf a (K)$,
\item if $v$ is any $\zz$-valuation on $K$ and $x \in \mf a (K)$, then either $v(x) \geq 0$ or $v(x)$ is even.
\end{itemize}
\end{lem}
\begin{lem}\label{L:Frob-single-lem}
Let $K$ be an algebraically closed field of characteristic $2$.
Let $C$ be a smooth projective curve of genus $\mf g$ over $K$ and $f \in K(C) = F$ non-constant.
Let $F_1, \ldots, F_r \in k[X]$ be separable, pairwise coprime polynomials of degree $d$.
If $F_j(f)^{-1} \in \mf a (F)$ for each $1 \leq j \leq r$, then $r < \frac{1}{d}(4+4\mf g)$.
\end{lem}
\begin{proof}
The proof idea is exactly as in \cite[Lemma 2.2]{Pasten_FrobeniusOrbits}.
Since $F_j(f)^2 = F_j(f^2)$ and in view of \Cref{L:Frobenius-trick}, we may without loss of generality assume that $f$ is not a square in $F$.
Consequently, $f : C \to \mbb P^1_K$ is a separable non-constant morphism.

Denote by $q_1, \ldots, q_n$ the different roots of the polynomials $F_j$; by assumption we have $n = rd$.
For a closed point $\mf p \in C$, denote by $e_f(\mf p)$ the ramification index of $\mf p$ under $f$.
We compute that
\begin{align*}
\sum_{j=1}^r \# (F_j(f))^{-1}(0) &= \sum_{i=1}^n \# f^{-1}(q_i) = \sum_{i=1}^n \left(\deg(f) - \sum_{\mf p \in f^{-1}(q_i)} (e_f(\mf p) - 1)\right) \\
&= n\deg(f) - \sum_{i=1}^n \sum_{\mf p \in f^{-1}(q_i)} (e_f(\mf p) - 1) \\
&\geq n\deg(f) - \sum_{P \in C} (e_{f}(\mf p) - 1) \geq (n-2)\deg(f) + 2 - 2\mf g,
\end{align*}
where the second equality uses \cite[Proposition II.2.6]{Silverman_EllCurves} and the last inequality uses the Hurwitz formula \cite[Theorem II.5.9]{Silverman_EllCurves}.
On the other hand, by \Cref{L:Frobenius-trick}, any zero of $F_j(f)$ must have multiplicity at least $2$, so
$$ \sum_{j=1}^r \# (F_j(f))^{-1}(0) \leq \frac{1}{2}\sum_{j=1}^r \deg(F_j(f)) = \frac{rd}{2}\deg(f)$$
from which we deduce that
$$ rd \leq 4 + \frac{4\mf g - 4}{\deg(f)} < 4 + 4\mf g$$
as desired.
\end{proof}

\begin{proof}[Proof of \Cref{thm:Frob-orbit-2}.]
First suppose that $f = g^{2^s}$ for some $s \in \nat$.
Note that $f$ and $g$ are then both not constant, so that $F_j(f) \neq 0$ for all $j = 1, \ldots, M$.
Since for $j \in \nat$ we have that $F_j(g)^{-1} + F_j(g)^{-1} = 0 \in \mf a (K)$ and $F_j(g)^{-1} + F_j(g^2)^{-1} = F_j(g)^{-1} + F_j(g)^{-2} \in \mf a (F)$, one sees by induction on $s$, using \Cref{L:Frobenius-trick}, that $F_j(f)^{-1} + F_j(g)^{-1} \in \mf a (K)$. Multiplying by $F_j(f)F_j(g)$, we obtain the existence of $h_j \in F$ with $F_j(f) + F_j(g) = F_j(f)F_j(g)(h_j^2 + h_j)$.
By symmetry, the same conclusion can be derived when $g = f^{2^s}$ for some $s \in \nat$.
This concludes the proof of one implication.

For the converse, we may assume without loss of generality that $K$ is algebraically closed (base changing to an algebraic extension cannot increase the genus \cite[Chapter V, Theorem 5]{Chevalley-Functions}).
Using \Cref{L:Frob-single-lem} we may reduce to the situation where $f$ and $g$ are both non-constant.
The hypotheses then imply that $F_j(f)$ and $F_j(g)$ are non-zero, and $F_j(f)^{-1} + F_j(g)^{-1} \in \mf a (F)$ for all $j \in \{ 1, \ldots, M \}$.
We may further assume that $f$ and $g$ are both non-squares: if we had for example $f = \tilde{f}^2$ for some $\tilde{f} \in F$, then $F_j(f) = F_j(\tilde{f})^2$ for all $j \in \{ 1, \ldots, M \}$, and by \Cref{L:Frobenius-trick} we obtain $F_j(\tilde{f})^{-1} + F_j(g)^{-1} \in \mf a (F)$, whereby we may replace $f$ by $\tilde{f}$.
After these reductions, it remains to show that $f = g$.
For the sake of arriving at a contradiction, let us assume $f \neq g$.
We will follow the argument of \cite[Theorem 1.6]{Pasten_FrobeniusOrbits}.

Fix a smooth projective curve $C$ such that $F = K(C)$.
As before, we consider $f, g$ as functions $C \to \mbb P^1_K$.
Without loss of generality $\deg(f) \geq \deg(g)$.
Denote by $S_j \subseteq k$ the set of roots of $F_j$.
By the hypotheses on the polynomials $F_j$, the sets $S_j$ for $j=1, \ldots, M$ are pairwise disjoint and each contain $d$ elements.
Set $N = dM$ and let $q_1, \ldots, q_N$ denote the different elements of $\cup_j S_j$.
As in the proof of \Cref{L:Frob-single-lem}, considering $f$ as a separable non-constant morphism $C \to \mbb P^1_K$, Hurwitz' formula yields that
$$ \sum_{j=1}^N \# f^{-1}(q_j) \geq (N-2)\deg(f) + 2 - 2\mf g.$$
Let $P$ be the set of closed points of $C$ where $f$ or $g$ has a pole, set
$$A = \lbrace \mf p \in C \mid \exists i \in \lbrace 1, \ldots, M \rbrace : (f(\mf{p}), g(\mf{p})) \in S_i \times S_i \rbrace $$
and let $B = A \cup P$.
For $\mf p \in C$ and $x \in F$, denote by $v_{\mf p}(x)$ the order of the root/pole of $x$ at $\mf p$, and $v_{\mf p}^+(x) = \max \lbrace v_{\mf p}(x), 0 \rbrace$.
We compute that
\begin{align*}
\sum_{j=1}^N \# f^{-1}(q_j) &= \sum_{j=1}^N \sum_{\mf p \in C} \min \lbrace 1, v_{\mf p}^+(f - q_j) \rbrace = \sum_{\mf p \in C} \sum_{j=1}^N \min \lbrace 1, v_{\mf p}^+(f - q_j) \rbrace \\
&\leq \sum_{\mf p \in C \setminus B} \sum_{j=1}^N \min \{ 1, v_{\mf p}^+(f - q_j)) \} + \# A + \# P;
\end{align*}
the last inequality follows because for $\mf p \in B$ the value of $\sum_{j=1}^N \min \{ 1, v_{\mf p}^+(f - q_j)) \}$ is at most $1$.
We have $\# P \leq \deg(f) + \deg(g) \leq 2\deg(f)$, and by \cite[Lemma 2.1]{Pasten_FrobeniusOrbits} we also obtain
$$ \# A \leq \left( 2 + 4 \sum_{j=1}^{\lceil (d-1)/2 \rceil} 2^j \right) \deg(f).$$
We now observe that, for $j \in \lbrace 1, \ldots, N \rbrace$ and $\mf p \in C \setminus B$, $v_{\mf p}(f-q_j) \neq 1$.
Indeed, suppose that $v_{\mf p}(f - q_j) = 1$ and $q_j \in S_i$.
The choice of our set $A$ implies that $v_{\mf p}(g - q) = 0$ for all $q \in S_i$, and we also have $v_{\mf p}(f - q) = 0$ for $q \in S_i \setminus \lbrace q_j \rbrace$.
We then compute that
$$ v_{\mf p}(F_i(f)^{-1} + F_i(g)^{-1}) = v_{\mf p}(\prod_{q \in S_i} (f - q)^{-1} + \prod_{q \in S_i} (g- q)^{-1}) = -1,$$
but the hypothesis implies that $F_i(f)^{-1} + F_i(g)^{-1} \in \mf a (F)$, contradicting the third part of \Cref{L:Frobenius-trick}.
We conclude that indeed, for $j \in \lbrace 1, \ldots, N \rbrace$ and $\mf p \in C \setminus B$, $v_{\mf p}(f-q_j) \neq 1$.
From this and the definition of the set $P$ we obtain that, for $j \in \lbrace 1, \ldots, N \rbrace$ and $\mf{p} \in C \setminus B$, $\min \lbrace 1, v_{\mf p}^+(f - q_j) \rbrace \leq \frac{1}{2}v_{\mf p}^+(f - q_j)$.
Putting all of our bounds together, we compute the following:
\begin{align*}
(N-2)\deg(f) + 2 - 2\mf g &\leq \sum_{j=1}^N \# f^{-1}(q_j) \\
&\leq \sum_{\mf p \in C \setminus B} \sum_{j=1}^N \min \{ 1, v_{\mf p}^+(f - q_j)) \} + \# A + \# P \\
&\leq \frac{1}{2}\sum_{\mf p \in C \setminus B} \sum_{j=1}^N v_{\mf p}^+ (f - q_j) + \left( 4 + 4 \sum_{j=1}^{\lceil (d-1)/2 \rceil} 2^j \right)\deg(f) \\
&\leq \left(\frac{N}{2} + 4 + 4 \sum_{j=1}^{\lceil (d-1)/2 \rceil} 2^j \right)\deg(f),
\end{align*}
where in the last step we used that $\sum_{\mf p \in C} v_{\mf p}^+(f - q_j) = \deg(f - q_j) = \deg(f)$ for all $j$.
Using that $N = dM$ and rearranging, this yields
$$ M < \frac{1}{d}\left( 4\mf g + 12 + 8\sum_{j=1}^{\lceil (d-1)/2 \rceil} 2^j\right),$$
finally contradicting the assumption on $M$.
We infer that $f = g$ as desired.
\end{proof}

\section{Existentially defining the orbit of the Frobenius endomorphism}\label{S:def-Frobenius}
We now complete the proof of \Cref{TI:Frobenius} (\Cref{T:define-Pn} below).
\begin{lem}\label{L:enough-irreducibles}
Let $\mf g \geq 0$ and let $d \geq 2\log_2(16 + \sqrt{8\mf g + 248})$ be a prime number.
For any prime number $p$ there exist $M(\mf g, d, p)$ pairwise coprime monic irreducible degree $d$ polynomials over $\ff_p$.
\end{lem}
\begin{proof}
See \cite[109]{Pasten_FrobeniusOrbits} for details of the computation in the case $p \geq 3$; we sketch the steps for the computation when $p=2$.

By elementary algebra, the assumption $d \geq 2\log_2(16 + \sqrt{8\mf g + 248})$ is easily seen to be equivalent to
$$ 2^{d-1} > 4\mf g - 4 + 16 \cdot 2^{\frac{d}{2}}$$
and in particular implies $d \geq 10$.
From this we compute that
\begin{align*}
\frac{2^d - 2}{d} &\geq \frac{2^{d-1}}{d} + 1 > (4\mf g - 4 + 16 \cdot 2^\frac{d}{2}) + 1 \\
&\geq \left(4\mf g +12 + 8 \sum_{j=1}^{\lfloor (d-1)/2 \rfloor} 2^j \right) + 1 > M(\mf g, d, 2).
\end{align*}
Since $d$ is prime, $\frac{2^d - 2}{d}$ is precisely the number of monic irreducible degree $d$ polynomials over $\ff_2$: indeed, every irreducible degree $d$ polynomial over $\ff_2$ has $d$ distinct roots in $\ff_{2^d} \setminus \ff_2$, and conversely, every element of $\ff_{2^d} \setminus \ff_2$ is the root of a unique monic irreducible degree $d$ polynomial over $\ff_2$.
Thus the desired statement follows.
\end{proof}

\begin{lem}\label{L:define-Frob-assymetric}
Let $\mf g \geq 0$.
There exists a positive-existential $\Lar$-formula $\phi_{\mf g}(x, y)$ such that for every field $K$ with $\charac(K) = p > 0$, for every function field in one variable $F/K$ of genus at most $\mf g$, and for any $f, g \in F$, one has the following.
\begin{itemize}
\item If $F \models \phi_{\mf g}(f, g)$, then either $f$ and $g$ are algebraic over $K$, or there exists $s \in \nat$ such that either $f = g^{p^s}$ or $g = f^{p^s}$.
\item If there exists $s \in \nat$ such that either $f = g^{p^s}$ or $g = f^{p^s}$, then $F \models \phi_{\mf g}(f, g)$
\end{itemize}
\end{lem}
\begin{proof}

Fix a prime number $d \geq 2\log_2(16 + \sqrt{8\mf g + 248})$.
By \Cref{L:enough-irreducibles} combined with the Chinese Remainder Theorem, we may find monic degree $d$ polynomials $F_1, \ldots, F_M \in \zz[X]$ with $M = M(\mf g, d, p)$ such that, for every prime $p < 4\mf g + 12$, the reductions of $F_1, \ldots, F_m$ modulo $p$ are pairwise distinct irreducible polynomials.

Define $\phi_{\mf g}(x, y)$ to be 
$$(\phi^{(2)}_{\mf g} \wedge 2 = 0) \vee \bigvee_{p=3}^{4\mf g + 12} (\phi^{(p)}_{\mf g} \wedge p=0) \vee (\phi^{(0)}_{\mf g} \wedge \bigwedge_{p=2}^{4\mf g + 12}(p \neq 0))$$ 
where
\begin{align*}
\phi^{(2)}_{\mf g}(x,y) &:\quad
\bigwedge_{j=1}^M \exists h_j ( F_j(x) + F_j(y) = F_j(x)F_j(y)(h_j^2 + h_j)), \\
\phi^{(p)}_{\mf g}(x,y) &:\quad
\bigwedge_{j=1}^M \exists h_j ( F_j(x)F_j(y) = h_j^2 ), \\
\phi^{(0)}_{\mf g}(x,y) &:\quad
\bigwedge_{j=1}^{\mf 4\mf g + 12} \exists h_j ( (x - j)(y - j) = h_j^2 ).
\end{align*}
Note that $\phi_{\mf g}$ is positive-existential; the inequality $p \neq 0$ may be replaced by the equivalent $\exists z (z \cdot p = 1)$.

Observe that if $p > 4\mf g + 12 = M(\mf g, 1, p)$, then the polynomials $X - j$ for $j=1, \ldots, 4\mf g + 12$ are pairwise distinct and irreducible in $\ff_p$.
The formula $\phi_{\mf g}$ is thus as desired: for fields of characteristic $\geq 3$ this follows from \Cref{thm:Frob-orbit-p} (see \cite[109]{Pasten_FrobeniusOrbits}) and for fields of characteristic $2$ this similarly follows from \Cref{thm:Frob-orbit-2}.
\end{proof}

\begin{thm}\label{T:define-Pn}
Let $\mf g \geq 0$.
There exists a positive-existential $\Lar$-formula $\pi_{\mf g}(x, y, z)$ with the following property:
given any prime $p$, any field $K$ of characteristic $p$, any function field $F/K$ in one variable of genus $\leq \mf g$, and any $t \in F$ which is transcendental over $K$ and not a $p$-th power in $F$, we have that for any $f, g \in F$,
$$ F \models \pi_{\mf g}(f, g, t) \quad\Leftrightarrow\quad \exists s \geq 0 : f = g^{p^s}.$$
\end{thm}
\begin{proof}
Define the formula $\pi_{\mf g}(x, y, z)$ to be
$$ \exists u \colon \phi_{\mf{g}}(u,z) \wedge \phi_{\mf{g}}(x,y) \wedge \phi_{\mf{g}}(ux, zy) \wedge \phi_{\mf g}(u(x+1),z(y+1)) $$
where $\phi_{\mf g}(x, y)$ is as in \Cref{L:define-Frob-assymetric}.

Let $F/K$ be a function field of genus $\leq \mf g$ over a field of characteristic $p$. We claim that, for $f, g, t \in K$ with $t$ not a $p$-th power, we have that
$$ K \models \pi_{\mf g}(f, g, t) \quad\Leftrightarrow\quad \exists s \geq 0 : f = g^{p^s},$$
which would show that $\pi_{\mf g}$ is as desired.
On the one hand, if $f = g^{p^s}$ for some $s \in \nat$, then $t^{p^s}f = (tg)^{p^s}$ and $t^{p^s}(f+1) = (t(g+1))^{p^s}$, so $F \models \pi_{\mf g}(f, g, t)$ (taking $t^{p^s}$ for $u$).

For the converse, assume that $F \models \pi_{\mf g}(f, g, t)$.
Since $t$ was assumed to be transcendental and not a $p$-th power, the properties of $\phi_{\mf g, p}$ imply that there exists $s \in \nat$ with $F \models \phi_{\mf g, p}(f, g) \wedge \phi_{\mf g, p}(t^{p^s}f, tg) \wedge \phi_{\mf g, p}(t^{p^s}(f+1), t(g+1))$.
If $f$ and $g$ are both constant, then $tg$ is not constant, and it follows from $F \models \phi_{\mf g, p}(t^{p^s}f, tg)$ that $f = g^{p^s}$ (e.g.~by comparing multiplicities of any zero of $t$).
Assume from now on that $f$ and $g$ are not both constant.
It follows from  $F \models \phi_{\mf g, p}(f, g)$ that $f$ and $g$ are then both non-constant and either $f = g^{p^k}$ for some $k \in \nat$, or $g = f^{p^k}$ for some $k \in \nat$.

Assume that $g = f^{p^k}$ for some $k \in \nat$; it remains to show that $k = 0$.
For an element $h \in F$, let us denote by $\div(h)$ the divisor of $h$.
The fact that $F \models \phi_{\mf g, p}(t^{p^s}f, tg) \wedge \phi_{\mf g, p}(t^{p^s}(f+1), t(g+1))$ implies that there exists $l_1, l_2 \in \zz$ (possibly negative) with $\div(t^{p^s}f) = p^{l_1}\div(tg)$ and $\div(t^{p^s}(f+1)) = p^{l_1}\div(t(g+1))$.
Using that $g = f^{p^k}$ and rearranging, we obtain
\begin{align*}
(p^s - p^{l_1})\div(t) &= (p^{k+{l_1}} - 1)\div(f) \quad\text{and} \\
(p^s - p^{l_2})\div(t) &= (p^{k+l_2} - 1)\div(f+1).
\end{align*}
If $s = l_1$, then since $\div(f) \neq 0$, this forces $k+l_1=0$, which (since both $k$ and $l_1 = s$ are non-negative) implies $k=l_1=s=0$, which we had to show.
A similar argument covers the case $s = l_2$.
So suppose now that $s \neq l_1$ and $s \neq l_2$.
It follows that there is a non-trivial $\zz$-linear combination of $\div(f)$ and $\div(f+1)$ equal to $0$, so that there exist non-zero $m \in \zz$ and $n \in \zz \setminus \lbrace 0 \rbrace$ such that $f^m(f+1)^n$ is constant; this is only possible when $f$ itself is constant, which we assumed not to be the case.
This concludes the proof of the claim. 
\end{proof}
\begin{opm}
\Cref{T:define-Pn} was proven in \cite[Theorem 1.5]{Pasten_FrobeniusOrbits} for $p > 2$ and under the slightly stronger hypothesis that $t$ is a uniformiser of some $K$-trivial valuation on $F$.
However, Philip Dittmann pointed out in conversation with Hector Pasten and the author that the formula presented there does not quite work: instead of $\pi_\mf g(x, y, z)$ as constructed in the proof of \Cref{T:define-Pn}, on \cite[110]{Pasten_FrobeniusOrbits} the formula
$$ \exists u \colon \phi_{\mf{g}}(u,z) \wedge \phi_{\mf{g}}(x,y) \wedge \phi_{\mf{g}}(ux, zy) $$
is proposed, which is satisfied by $(t, t^p, t)$ (taking $u = t^p$) and does thus not serve the desired purpose.
The modification of the definition of $\pi_{\mf g}$ in the proof of \Cref{T:define-Pn} resolves this issue.
\end{opm}

\begin{cor}\label{C:define-Pn}
Let $K$ be a field of characteristic $p > 0$, let $F/K$ be a function field, and fix any transcendental element $t \in F$.
There exists a positive-existential $\Lar$-formula $\pi(x, y, z)$ such that for all $f, g \in F$ one has
$$ F \models \pi(f, g, t) \quad\Leftrightarrow\quad \exists s \in \nat : f = g^{p^s}.$$
\end{cor}
\begin{proof}
Let $d$ be the transcendence degree of $F/K$, let $t_1, \ldots, t_d$ be a transcendence basis of $F$ over $K$ with $t_d = t$.
Replacing $K$ by $K(t_1, \ldots, t_{d-1})$, we may reduce to the case where $F/K$ is a function field in one variable.

Let $\mf g$ be the genus of $F/K$.
Let $m \in \nat$ be maximal with the property that $t$ is a $p^m$-th power in $F$.
Let $\pi_{\mf g}$ be the formula from \Cref{T:define-Pn}.
It suffices to set $\pi(x, y, z)$ equal to
$\exists w (t = w^{p^m} \wedge \pi_{\mf g}(f, g, w))$.
\end{proof}

\section{A reduction theorem for undecidability}\label{S:undecidability}
With \Cref{T:define-Pn} in hand, we can show the following undecidability result for arbitrary function fields.
It reduces the problem of showing that $\Th_{\exists}(F, R)$ is undecidable for a function $F$ and a well-chosen ring of coefficients $R$, to the existential definability of some well-behaved subset.
The results of this section should be seen as a refinement (by keeping track of parameters) of the reduction technique developed in \cite[Lemma 1.5]{Shl96} and \cite[Section 2]{EisShlap17} and going back to the original work of Pheidas.

For a valuation $v$ on a field $F$, we denote by $\mc{O}_{v}$ the corresponding valuation ring, by $vF$ the value group, and by $Fv$ the residue field.
\begin{thm}\label{T:Pheidas}
Let $K$ be a field of characteristic $p > 0$, let $F$ be a function field in one variable over $K$.
Consider the first-order language $\La = \lbrace +, -, \cdot, 0, 1, t, \mc{O} \rbrace$, where $t$ is a constant symbol and $\mc{O}$ a unary relation symbol.
Consider $F$ as an $\La$-structure by interpreting $+, -, \cdot, 0, 1$ in the obvious way, and such that for some valuation $v$ on $F$, $t^F$ is transcendental over $K$, and
$$ \ff_p(t^F) \neq \mc{O}_{v} \cap \ff_p(t^F) \subseteq \mc{O}^F \subseteq \mc{O}_{v}.$$
Then the positive-existential $\La$-theory of $F$ is undecidable.
\end{thm}
\begin{proof}
In the language $\LaP^p = \lbrace 0, 1, +, \mid_p \rbrace$, consider $\nat$ as a structure by letting $+^{\nat}, 0^\nat, 1^\nat$ be the usual addition, zero element, and identity element, and defining for $n, m \in \nat^+$ the relation $n \mid_p^{\nat^+} m \Leftrightarrow m = np^s$ for some $s \in \nat$.
It was shown in \cite[Theorem 1]{Pheidas87} that the positive-existential $\LaP^p$-theory of $\nat$ is undecidable.
To prove the Theorem, it thus suffices to show that there is an effective interpretation of the positive-existential $\LaP^p$-theory of $\nat$ in the positive-existential $\La$-theory of $K$.

The assumption $\ff_p(t^F) \neq \mc{O}_{v} \cap \ff_p(t^F)$ implies that there exists $\tilde{t} \in \ff_p(t^F)$ with $v(\tilde{t}) > 0$.
By replacing $t^K$ with $\tilde{t}$, we may assume from now on that $v(t^F) > 0$.
Write $s = v(t^F)$.

Now consider an arbitrary positive-existential $\LaP^p$-sentence $\varphi$.
After applying formal manipulations, we may assume without loss of generality that $\varphi$ is of the form $\exists n_1, \ldots, n_l \psi$ where $\psi$ is a conjunction of unnested atomic formulas.
Let $\varphi'$ be the $\La$-sentence
$$\exists x_1, \ldots, x_l, y_1, \ldots, y_l, z (\psi' \wedge t \cdot z = 1 \wedge \bigwedge_{i=1}^l (x_i \cdot y_i = 1) \wedge \bigwedge_{i=1}^l \mc O (x_i))$$
where $\psi'$ is obtained from $\psi$ be replacing every atomic $\LaP^p$-formula in the conjunction by an atomic $\La$-formula according to the following rules:
\begin{align*}
n_i = n_j &\qraq \mc O (x_i \cdot y_j) \wedge \mc O (x_j \cdot y_i) \\
n_i = 0 &\qraq \mc O (y_i) \\
n_i = 1 &\qraq \mc O (x_i^s \cdot z) \wedge \mc O (y_i^s \cdot t) \\
n_i + n_j = n_k &\qraq \mc O ((x_i \cdot x_j) \cdot y_k) \wedge \mc O ((y_i \cdot y_j) \cdot x_k) \\
n_i \mid^p n_j &\qraq \exists \tilde{x_j}, \tilde{y_j} (\pi (\tilde{x_j}, x_j, t) \wedge \tilde{x_j} \cdot \tilde{y_j} = 1 \wedge \mc O (\tilde{x_j} \cdot y_i)) \wedge \mc O (\tilde{y_j} \cdot x_i))
\end{align*}
Here, $\pi$ is the formula from \Cref{C:define-Pn}.
One sees that $F \models \varphi'$ if and only if $\nat \models \varphi$.
Since $\varphi'$ is a positive-existential $\La$-sentence, the undecidability of the positive-existential $\La$-theory of $K$ follows.
\end{proof}
This Theorem will be applied in the following form:
\begin{cor}\label{C:Pheidas}
Let $K$ be a field of characteristic $p > 0$, let $F$ be a function field in one variable over $K$.
Let $\La$ be a first-order language containing the symbols $+, -, \cdot, 0, 1, t$, where $t$ is a constant symbol.
Consider $F$ as an $\La$-structure by interpreting $+, -, \cdot, 0, 1$ in the obvious way, and such that for some valuation $v$ on $F$, $t^F$ is transcendental over $K$, and there exists a positive-existentially $\La$-definable subset $\mc{O}$ of $F$ such that
$$ \ff_p(t^F) \neq \mc{O}_{v} \cap \ff_p(t^F) \subseteq \mc{O} \subseteq \mc{O}_{v}.$$
Then the positive-existential $\La$-theory of $F$ is undecidable.
\end{cor}

\section{Proof of the main theorem}\label{S:final}
Consider a function field in one variable $F/K$ with arbitrary base field $K$ and let $t \in F$ be transcendental over $K$.
In view of \Cref{C:Pheidas}, to show that $\Th_{\exists}(F, \zz[t])$ is undecidable, it remains to find some existentially definable subset $\mc{O}$ in $F$ as described there.
In particular, in the case of a global function field (i.e.~with $K$ a finite field), the hypotheses on $\mc{O}$ demanded in \Cref{C:Pheidas} are satisfied when $\mc{O}$ is a positive-existentially definable non-trivial valuation ring on $F$.
We emphasize that the positive-existential definability of valuation rings in global function fields has been known for a long time if one allows arbitrary parameters, see \cite[Lemma 3.22]{ShlapentokhGlobal} for function fields of odd characteristic, or \cite[Theorem 5.15]{EisentragerThesis} for a proof including characteristic $2$.
Thus, the only new challenge taken up here is to obtain this positive-existential definability without introducing any additional parameters beyond the element $t$.
This will be achieved in \Cref{T:defining-valuation}.

The reader should compare the results of this section with \cite[Section 6]{EisShlap17} where corresponding positive-existential definability and decidability results are obtained under essentially identical hypotheses, but with less control on the required parameters in the language.

\begin{lem}\label{L:purely-inseparable}
Let $(K, v)$ be a $\zz$-valued field of characteristic $p > 0$ with uniformiser $t$.
Let $L/K$ be a purely inseparable extension of degree $p^m$.
Then $v(K(t^{1/p^m})) = v(L(t^{1/p^m})) = \frac{1}{p^m}\zz$.
\end{lem}
\begin{proof}
We consider every purely inseparable extension of $K$ endowed with the unique extension of $v$ (see \cite[Corollary 3.2.10]{Eng05}), which we also denote by $v$.

Since $v(t^{1/p^m}) = \frac{1}{p^m}v(t) = \frac{1}{p^m}$, we must have that $v(K(t^{1/p^m})) = \frac{1}{p^m}\zz$ by the Fundamental Inequality \cite[Theorem 3.3.4]{Eng05}.
We also have $v(K^{1/p^m}) = \frac{1}{p^m}\zz$, as the map $x \mapsto x^{p^m}$ defines an isomorphism (of valued fields) $K^{1/p^m} \to K$.
Since $K(t^{1/p^m})$ and $K^{1/p^m}$ have the same value group, we finally conclude that also the intermediate field $L(t^{1/p^m})$ must have value group $\frac{1}{p^m}\zz$.
\end{proof}

\begin{thm}\label{T:defining-valuation}
Let $K$ be a field of characteristic $p > 0$ not containing an algebraically closed subfield, and $F/K$ a function field in one variable.
Consider $F$ as a structure in the language $\La = \{ +, -, \cdot, 0, 1, t \}$ where $t^F$ is transcendental over $K$.
Then there exists a positive-existentially $\La$-definable subset $\mc{O}$ of $F$ such that
$$\ff_p(t^F) \neq \mc{O}_v \cap \ff_p(t^F) \subseteq \mc O \subseteq \mc O_v$$
for some $K$-trivial $\zz$-valuation $v$ on $F$.
\end{thm}
\begin{proof}
Let us write $t$ for $t^F$.
Let $F_s$ denote the relative separable closure of $K(t^F)$ in $F$.
By \cite[Theorem V.5.3 and Lemma V.8.3]{Chevalley-Functions}, all but finitely many $\zz$-valuations on $\ff_p(t)$ extend to a $\zz$-valuation on $F_s$.
We may thus fix a $\zz$-valuation $v$ on $F_s$ and $\tilde{t} \in \ff_p(t)$ such that $\tilde{t}$ is a uniformiser for $v$.
Replacing $t$ by $\tilde{t}$, we assume from now on that $t$ is a uniformiser of $v$ (on $F_s$).

We use \Cref{L:purely-inseparable} to find $m \in \nat$ such that $v(F(t^{1/p^m})) = v(F_s(t^{1/p^m}))$.
Applying \Cref{L:interpretation-argument} we may replace $F$ by $F(t^{1/p^m})$ and $t$ by $t^{1/p^m}$ and thus assume from now on that $t$ is a uniformiser for the valuation $v$ on $F$.

Consider the residue field $Fv$ of $v$, which is a finite extension of $K$.
In particular, the hypotheses on $K$ imply that $Fv$ does not contain an algebraically closed subfield.
It follows that there exists some prime number $q$ and a finite extension $\ell / \ff_p$ contained in $Fv$ such that $\ell$ has a finite Galois extension of degree $q$ not contained in $Fv$.
The field $F' = F \otimes_{F \cap \ell} \ell$ is a function field in one variable over $\ell$, and the valuation $v$ on $F$ uniquely extends to a $\zz$-valuation on $F'$ (which we also denote by $v$).
We may again invoke \Cref{L:interpretation-argument} to replace $F$ by $F'$, and so we may assume without loss of generality that $\ell \subseteq F$.

Let $f \in \ell[T]$ be an irreducible degree $q$ polynomial.
Since $\ell/\ff_p$ is separable and hence generated by a single element, we may write $f(T) = g(\alpha, T)$ for some $g \in \ff_p[A, T]$, $\alpha \in \ell$.
Let $h(T)$ be the minimal polynomial of $\alpha$ over $\ff_p$ and denote by $\alpha_1, \ldots, \alpha_k$ the conjugates of $\alpha = \alpha_1$.
The splitting field $G$ of $f$ over $F$ is identical to the splitting field of $g(\alpha_i, T)$ for every $i \in \lbrace 1, \ldots, k \rbrace$ and is a cyclic extension of $F$ of degree $l$; let $\sigma$ be a generator of its Galois group.
We may consider the cyclic algebra $\mc A = (G, \sigma, t)$ which is by definition generated as an $F$-algebra by $G$ and an element $y$ subject to the relations $y^q = t$ and $xy = y\sigma(x)$ for all $x \in G$.
In view of the previous discussion, by existentially quantifiying over the roots of $h(T)$ when describing $G$, we have that $\mc A$ is positive-existentially $\La$-interpretable in $F$ without parameters.

Write $\Delta \mc A$ for the set of $K$-trivial $\zz$-valuations $w$ on $F$ for which $\mc A \otimes_F F^w$ is a division algebra, where $F^w$ denotes the completion.
We have that $v \in \Delta \mc A$: by \cite[\nopp 15.1, Lemma]{Pie82} $\mc A \otimes_F F^v$ is a division algebra if and only if $t$ is not a norm of the extension $GF^v/F^v$, but since this is an unramified extension of degree $q$, all norms must have $v$-value divisible by $q$.
Denoting by $\Trd$ and $\Nrd$ the reduced trace and reduced norm operations on $\mc A$, respectively, we obtain that $\Trd(x) \in \mc{O}_v$ for any $x \in \mc A$ with $\Nrd(x) = 1$, as the latter condition forces $x$ to be integral over $\mc{O}_v$: either $x^q = 1$ and this is immediate, or $F^v(x)/F^v$ is a degree $q$ field extension contained in $\mc A \otimes_F F^v$ and $1 = \Nrd(x) = \operatorname{N}_{F^v(x)/F^v}(x)$ \cite[Lemma 2.3]{Dit17}, in which case $x$ lies in the valuation ring of the unique extension of $v$ from $F^v$ to $F^v(x)$ \cite[Theorem 14:1]{OMe00} and is hence integral over $\mc O_v$.
Combining this with \cite[Proposition 2.9]{Dit17} we get
$$ \bigcap_{w \in \Delta \mc A} \mc{O}_{w} \cap \ff_p(t^F) \subseteq T(\mc A) \subseteq  \mc{O}_{v} $$
where 
$$T(\mc A) = \lbrace \Trd(x) + \Trd(y) \mid x, y \in \mc A, \Nrd(x) = \Nrd(y) = 1 \rbrace,$$ which is positive-existentially $\La$-definable.
By Strong Approximation, we may find $y \in \ff_p(t)$ such that $v(y) > 0$ and $w(y) \leq 0$ for $w \in \Delta \mc A$ with $w\vert_{\ff_p(t)}$ not equivalent to $v\vert_{\ff_p(t)}$.
Now define the positive-existentially $\La$-definable set
$$S = \lbrace x \in F \mid \exists \alpha \in F : x^q g(\alpha, yx)^{-1} \in T(\mc A), h(\alpha) = 0 \rbrace.$$
Since for any $w \in \Delta \mc A$, $x \in F$, and $\alpha \in F$ with $h(\alpha) = 0$ we have $w(g(\alpha, x)) = w(f(x)) = \min \{ 0, q w(x) \}$, we compute that 
$$ \mc{O}_{v} \cap \ff_p(t^K) \subseteq S \subseteq \mc{O}_{v}$$
as desired.
\end{proof}

\begin{thm}\label{T:ES-constants}
Let $K$ be a field of characteristic $p > 0$ not containing an algebraically closed subfield, $F/K$ a function field.
Consider $F$ as a structure in the language $\La = \lbrace +, -, \cdot, 0, 1, t \rbrace$, where $t^F \in F$ is transcendental over $K$.
Then the positive-existential $\La$-theory of $F$ is undecidable.
\end{thm}
\begin{proof}
Let $d$ be the transcendence degree of $F$ over $K$ and let $t_1, \ldots, t_d$ be a transcendence basis of $F$ over $K$ with $t_d = t^F$.
Replacing $K$ by $K(t_1, \ldots, t_{d-1})$, we may reduce without loss of generality to the case where $F/K$ is a function field in one variable.

The statement now follows directly from \Cref{C:Pheidas} and \Cref{T:defining-valuation}.
\end{proof}
\noindent
That \Cref{T:ES-constants} is a rephrasing of \Cref{TI:ES-constants} was explained at the end of \Cref{S:definability}.

\printbibliography
\end{document}